\documentclass[11pt]{amsart}
\usepackage{amssymb,amscd}
\usepackage{graphicx}
\DeclareGraphicsExtensions{.jpg, .eps}
\usepackage{textcomp}
\newcommand{\li}{{\rm lim}\sp{1}}

\newcommand{\N}{{\mathbb N}}
\newcommand{\T}{{\mathcal T}}
\newcommand{\PP}{{\mathcal P}}
\newcommand{\Z}{{\mathbb Z}}
\newcommand{\Q}{{\mathbb Q}}
\newcommand{\R}{{\mathbb R}}

\newcommand{\XX}{{\mathbb X}}
\newcommand{\YY}{{\mathbb Y}}
\newcommand{\A}{{\mathcal A}}
\newcommand{\CC}{{\mathcal C}}

\newcommand{\Om}{{\Omega}}

\numberwithin{equation}{section}

 \newcommand{\TT}{\mathbb{T}}

 \newcommand{\lra}{\longrightarrow}
 \newcommand{\im}{\mbox{Im\,}}

\newtheorem{theorem}{Theorem}[section]
\newtheorem{lemma}[theorem]{Lemma}

\theoremstyle{definition}
\newtheorem{definition}[theorem]{Definition}
\newtheorem{prop}[theorem]{Proposition}
\newtheorem{cor}[theorem]{Corollary}

\newtheorem{example}[theorem]{Example}

\theoremstyle{remark}
\newtheorem{remark}[theorem]{Remark}
\title
{Tiling spaces, codimension one attractors and shape} \keywords{} \subjclass{}
\address{Department of Mathematics, University of Leicester, University Road, Leicester, UK}

\author {Alex Clark}\author{John Hunton}

\thanks{The University of Leicester funded study leave for both authors during the course of this research.}

\begin{document}

\begin{abstract}
We establish a close relationship between, on the one hand, expanding, codimension one attractors of diffeomorphisms on closed manifolds (examples of so-called {\em strange attractors}), and, on the other, spaces which arise in the study of {\em aperiodic tilings}. We show that every such orientable attractor is homeomorphic to a tiling space of either a substitution or a projection tiling, depending on its dimension. We also demonstrate that such an attractor is shape equivalence to a  $(d+1)$-dimensional torus with a finite number of points removed, or, in the non-orientable case, to a space with a 2 to 1 covering by such a torus-less-points. This puts considerable constraints on the topology of codimension one attractors, and constraints on which manifolds tiling spaces may be embedded in. In the process we develop a new invariant for aperiodic tilings, which, for one dimensional tilings is in many cases finer than the cohomological or $K$-theoretic invariants studied to date.
\end{abstract}

\maketitle

\section{Introduction}

This work establishes a close relationship between, on the one hand, expanding, codimension one attractors of diffeomorphisms on closed manifolds (examples of so-called {\em strange attractors}), and, on the other, spaces which arise in the study of {\em aperiodic tilings}.

Following the important programme initiated by Smale~\cite{Sm,R}, hyperbolic attractors of smooth diffeomorphisms have played a key role in understanding the structurally stable diffeomorphisms of closed, smooth manifolds. A $C^k$--diffeomorphism $h\colon M \to M$ of  a $C^k$--manifold ($k\geqslant1$) $M$ is \emph{structurally stable\/} if all diffeomorphisms sufficiently close to $h$ in the $C^k$--metric are topologically conjugate to $h.$  An attractor $A\subset M$ of $h$ is \emph{hyperbolic\/} if  the tangent bundle of the attractor admits an $h$--invariant continuous splitting $\mathbf{E}^s + \mathbf{E}^u$  into uniformly contracting $\mathbf{E}^s$ and expanding $\mathbf{E}^u$ directions. An important class of hyperbolic attractors are the \emph{expanding} attractors, those with the same topological dimension, say $d$,  as the fibre of  $\mathbf{E}^u.$ Expanding attractors locally have the structure of the product of a $d$--dimensional disk and a Cantor set~\cite{W} and are therefore sometimes referred to as $\emph{strange}$ attractors. Locally, the diffeomorphism $h$ expands the disks and contracts in the Cantor set direction. Here we shall focus on codimension one expanding attractors, {\em i.e.}, the case that $A$ is compact and connected (a \emph{continuum}) with topological dimension $d$ one less than the dimension $d+1$ of the ambient manifold $M.$

The tilings we have in mind are patterns in Euclidean space that admit no non-trivial translational symmetries, but nevertheless have the property that arbitrarily large compact patches of the pattern repeat themselves throughout the space. The Penrose tiling is perhaps the best known example of such a pattern, but the class is huge and rich, indeed infinite, and contains, for example, the geometric patterns used to model physical quasicrystals \cite{SD}. A standard tool in the study of any such pattern $P$ is the construction of an associated {\em tiling space\/} $\Omega_P$, a topological space whose points correspond to the set of all patterns locally indistinguishable from $P$. Topological properties of $\Omega_P$, in particular the \v{C}ech cohomology groups $H^*(\Omega_P)$ and various forms of $K$-theory, have long been known to contain  key geometric information about the original pattern $P$,  see, for example, \cite{BBJS, BHZ,CS}.

Our main aim is to describe the possible spaces $A$, both up to homeomorphism and up to shape equivalence, that can arise as a codimension one expanding attractor of a diffeomorphism.  In brief, we show that every such attractor is homeomorphic to a tiling space $\Omega_P$ for some $P$, but that the converse fails (in some sense, it fails in almost all cases). The shape equivalence description gives essentially a complete description of the possible cohomology rings of any such attractor $A$.

Our approach uses tools drawn from both shape theory and homological algebra and in doing so introduces a new invariant that gives an obstruction to the existence of a codimension one embedding of a tiling space in a manifold. Moreover, this invariant is finer than \v{C}ech cohomology and we give examples of  tilings with identical cohomology that it distinguishes.

In drawing on a diverse range of mathematical topics, it is perhaps not reasonable to assume the reader has specialist knowledge of expanding attractors, tiling spaces, shape theory or homological algebra; we introduce the necessary concepts or results directly, where possible. The ideas relating to shape theory and the homological algebra we use are presented in Section 2, while the details we assume of expanding attractors and tiling spaces are discussed in Section 3. The interested reader will find further background information on these topics in \cite{Btext, R, Sadbook}.

We detail our main results below; these are proved in Sections 4 and 5.

\subsection{Our main results}

Our initial results concern models for a codimension 1 expanding attractor up to {\em shape equivalence}. Shape equivalence here means equivalence in the {\em shape category}. We explain more about this notion in the next section, but for now we note that the shape category is a natural one to consider when analysing spaces which readily occur as inverse limits of topological spaces (such as both attractors and tiling spaces), but that shape equivalence is distinct from relations such as homeomorphism or homotopy equivalence. Nevertheless, two spaces that are shape equivalent necessarily share all the same {\em shape invariants}, which include \v{C}ech cohomology and certain forms of $K$-theory. Our identification of the shape of a codimension 1 attractor thus allows both ready computation of the \v{C}ech cohomology, {\em etc}., and  also puts considerable constraints on the possible cohomology rings that can arise.

Our first result shows that a codimension 1 expanding attractor is shape equivalent to a finite polyhedron of a very specific kind.

\begin{theorem}\label{thm1}
Let $A$ be a codimension 1 expanding attractor of the diffeomorphism $h\colon M\to M$. If $A$ is orientable, then it is shape equivalent to a $(d+1)$-dimensional torus with a finite number of points removed, $\TT^{d+1}\!\!-\{k\}$ say. If $A$ is unorientable, it is shape equivalent to a polyhedron that has a 2 to 1 cover by some $\TT^{d+1}\!\!-\{k\}$.
\end{theorem}

Fundamental to our work is Williams' foundational paper \cite{W} which shows that any continuum $A$ that occurs as an expanding attractor  (independent of
codimension) is homeomorphic to the inverse limit $\Lambda$ of a sequence

$$\cdots\to K\buildrel f\over\lra K\buildrel f\over\lra K$$
formed from a single map of a branched manifold $f\colon K \rightarrow K$
satisfying certain expanding properties, and the restriction of $h$
to $A$ is conjugate to the shift map of $\Lambda.$

However, our analysis of attractors splits into two cases, which display significantly different behaviours. On the one hand, in the case where $d=1$, and so $M$ is a closed surface, Williams' branched manifold can be taken as a one point union of copies of the circle, and so the shape theoretic analysis leads us to the study of endomorphisms of free groups, being the homotopy groups of these spaces. The higher dimensional cases, $d\geqslant 2$, involve far more complicated branched manifolds $K$ and a different approach is needed. Here work of Plykin \cite{P1,P2} comes to our aid.

Our second set of results, which follows from these analyses, establishes the connection between the codimension 1 oriented attractors and tiling spaces: again the cases $d=1$ and $d\geqslant 2$ must be treated separately. In the case $d=1$, each such oriented attractor is homeomorphic to a tiling space associated to some so-called {\em primitive substitution\/} tiling. This is well known to the experts and is mentioned in~\cite{BD1}, but we sketch the argument in Section \ref{subtil} for completeness. This argument however stands little chance of generalising to higher dimensions, and our main result for $d\geqslant2$ realises all such attractors up to homeomorphism as tiling spaces associated to a largely distinct class of tilings, the so-called {\em projection tilings}. We note equally, however, that this second approach does not satisfy the case $d=1$: we show that there are certainly 1 dimensional attractors which are substitution tiling spaces, but not projection tiling spaces of the sort needed for the higher dimensions.

\begin{theorem}\label{thm2}
Every oriented codimension 1 attractor $A$ in the $(d+1)$-dimensional manifold $M$ is homeomorphic to the tiling space $\Omega_P$ of an aperiodic tiling $P$ of $\R^d$. In the case $d=1$ we may choose $P$ to be given by a primitive substitution; for $d\geqslant2$, we can describe $P$ as a projection tiling.
\end{theorem}

We consider also the converse question: given a tiling space $\Omega_P$, can we realise it as a codimension 1 attractor for some suitable $M$ and $h$? In general the answer is `no'. In the case of higher dimensional manifolds, $(d+1)\geqslant3$, the shape theoretic result of Theorem \ref{thm1} puts such constraints on the cohomology ring $H^*(\Omega_P)$ for any tiling $P$ which models $A$ that most tilings are immediately ruled out as sources of models for codimension 1 attractors.

In the case $d=1$, cohomology is not sufficient to rule out potential models. However, our shape-theoretic analysis leads us to obtain a new invariant $L(\Omega_P)$ associated to a tiling space, whose vanishing is a necessary condition on realising $\Omega_P$ as a codimension 1 subspace of a manifold. This obstruction is comparable, though apparently quite distinct from, obstructions based on the topology of the asymptotic components \cite{HM}, but as with them its vanishing does not in general guarantee the existence of an embedding $\Omega_P\hookrightarrow M$.

The $L$-invariant also provides a new tool to distinguish tiling spaces, and in Section \ref{sect44} we exhibit examples which cannot otherwise be told apart using standard cohomological or $K$-theoretic calculations.

Finally, let us note that many of our results fail to be true if we ask about attractors of codimension greater than 1: this may easily be seen in the case of the classic Smale example of the dyadic solenoid, which occurs as an oriented, codimension 2 attractor in a 3-torus, but is not shape equivalent to any finite polyhedron, nor is it homeomorphic to any tiling space. In contrast to this, Anderson and Putnam \cite{AP} show that every substitution tiling space $\Omega_P$ of the type they consider has the structure of an expanding attractor for {\em some\/} smooth diffeomorphism of a smooth (possibly high dimensional) manifold, but the natural question of which manifolds $M$ in which such $\Omega_P$ can occur is as yet unanswered.

The organisation of this paper is follows. In section 2 we recall the basic facts about shape theory and shape equivalence that we need. This leads us also to introduce some related homological algebra, and in particular discuss aspects of the $\lim^1$ functor and its relationship to the concept of {\em movability}. In section 3 we introduce concepts and notations we use to discuss tiling spaces, attractors and their associated paraphernalia. In this section we define our $L$-invariant (in fact the first of a series of invariants for tiling spaces), and recall the results of Plykin \cite{P1,P2} we need in the final section.

In Section 4 we specialise to $d=1$ and begin by proving Theorem \ref{thm1} in this case. We show in Section \ref{shapeof1} that any codimension 1 attractor in a surface is shape equivalent to a one point union of a finite number of circles, and as such is determined by a finite rank free group $F$ and an automorphism $s\colon F\to F$. However, most such automorphisms are not realisable as expanding attractors in surfaces and we develop in Section \ref{Realising} our homological approach to aid computation of our main obstruction to an automorphism arising via an attractor. We sketch in Section \ref{subtil} how all such oriented attractors in surfaces can be realised as substitution tiling spaces (Theorem \ref{thm2}), and apply in Section \ref{sect44} our $L$-invariant and homological results to demonstrate examples of non-embedding tiling spaces and to distinguish aperiodic tilings indistinguishable by cohomology or $K$-theory.

In Section 5 we consider the rather different case $d\geqslant2$, proving Theorems \ref{thm1} and \ref{thm2} in these dimensions. Here we introduce the generalised projection spaces needed for realising attractors in higher dimensions $d\geqslant2$, but show that their analogues cannot account for all attractors when $d=1$. We conclude by showing that, in analogy with the case $d=1$, most projection tilings do not possess codimension 1 embeddings in $(d+1)$-manifolds; this follows from cohomological considerations and the shape theoretic result of Theorem \ref{thm1}.

\section{Shape theory}\subsection{The shape category, stability and movability.}
We sketch the basic notions and perspectives of the shape theory we use. Fuller details may be found in, for example, the books \cite{Btext, MS}.

We deal with two underlying categories of spaces. The first, $\T$, has as objects topological spaces, and morphisms the homotopy classes of maps. The second, $\PP$, is the full subcategory of $\T$ with objects those spaces which can be given the structure of a finite CW complex (`finite polyhedra' in the shape literature). In each case we will also need the corresponding categories of pointed spaces: each such space $X_n$ will then have a specified base point $x_n$, and all maps and homotopies will preserve base points. In general we shall suppress mention of the base point in our notation unless it is expressly needed.

For each of our categories, we also consider the corresponding pro-categories (see \cite{MS}, or even \cite{AM}, for the full definition) of diagrams of objects indexed by a directed set $D.$ For the cases we consider, we can always take $D=\N$, in which case, an object in the pro-category pro-$\CC$  of the category $\CC$ is a tower $$\XX:\qquad\cdots\to X_n\to X_{n-1}\to\cdots \to X_2\to X_1\to X_0$$
whose objects $X_n$ and maps $X_n\to X_{n-1}$ are in $\CC$. Morphisms in pro-$\CC$ are equivalence classes of commuting maps of towers which do not necessarily preserve levels (i.e., a map $\XX\to\YY$ consists of maps in $\CC$ running $X_{r(n)}\to Y_n$, for  $n \in \N$, making the corresponding diagram commute and with $r(n)\to\infty$ monotonically as $n\to\infty$). Two commuting maps of towers are equivalent if they induce the same map on the inverse limits of the towers. The category $\CC$ has a standard embedding as a subcategory of pro-$\CC$ given by identifying a $\CC$-object $X$ with the constant tower
$$\cdots \to X\buildrel 1\over\longrightarrow X\buildrel 1\over\longrightarrow \cdots\buildrel 1\over\longrightarrow  X\buildrel 1\over\longrightarrow X$$
in pro-$\CC$, and without further comment we shall identify objects in $\CC$ as objects in pro-$\CC$ in this manner.

The shape category arises from certain equivalences on such towers, and considers those objects in $\T$ which, up to  these equivalences, can be considered as objects in pro-$\PP$. Explicitly, we use the notion of a {\em $\PP$-expansion\/} of a space $X$ in $\T$, which is effectively a representation of $X$ by a tower of spaces drawn from the subcategory $\PP$.

\begin{definition} A {\em $\PP$-expansion\/} of an object $X\in\T\subset$ pro-$\T$ is a map $\alpha\colon X\to\XX$ in pro-$\T$ for some object $\XX$ in pro-$\PP$ with the universal property that for each morphism $h \colon X\to\YY$ with $h$ in pro-$\T$ and $\YY$ in pro-$\PP$, there is a unique map $f\colon\XX\to\YY$ in pro-$\PP$ factoring $h$ as $X\buildrel \alpha\over\longrightarrow\XX\buildrel h\over\longrightarrow\YY$.
\end{definition}

An key result for shape theory is that every object in $\T$ admits a $\PP$-expansion.

\vspace{.02in}

It is important to note that if $X$ is homeomorphic to the inverse limit $\displaystyle\lim_{\longleftarrow}\{\XX\}$ for some object $\XX$ in pro-$\PP$, then the universal map
$$X=\lim_{\longleftarrow}\{\XX\}\longrightarrow \:\:\:\:\cdots\to X_n\to X_{n-1}\to\cdots \to X_0$$
gives a $\PP$-expansion of $X$, but the converse does not generally hold: if $\alpha\colon X\to\XX$ is a $\PP$-expansion of $X$ then there is no general reason that $X$ is homeomorphic to $\displaystyle\lim_{\longleftarrow}\{\XX\}$.

As usual, if $\alpha\colon X\to\XX$ and $\alpha'\colon X\to\XX'$ are two $\PP$-expansions of $X$, there is a natural isomorphism $i\colon \XX\to\XX'$ in pro-$\PP$.

We need a corresponding notion of equivalence on morphisms.

\begin{definition}Suppose $\alpha\colon X\to\XX$ and $\alpha'\colon X\to\XX'$ are two $\PP$-expansions of some object $X\in\T$, with natural isomorphism $i\colon \XX\to\XX'$ in pro-$\PP$, and suppose $\beta\colon Y\to\YY$ and $\beta'\colon Y\to\YY'$ are two $\PP$-expansions of some object $Y\in\T$, with natural isomorphism $j\colon \YY\to\YY'$ in pro-$\PP$. Then two morphisms $f\colon\XX\to\YY$ and $f'\colon\XX'\to\YY'$ in pro-$\PP$ are equivalent, written $f\sim f'$, if
$$\begin{array}{ccc}
\XX&\buildrel i\over\longrightarrow&\XX'\\
f\big\downarrow\phantom{f}&&\phantom{f'}\big\downarrow f'\\
\YY&\buildrel j\over\longrightarrow&\YY'
\end{array}$$
commutes in pro-$\PP$.
\end{definition}

\begin{definition}
The {\em shape category\/} has objects the objects of $\T$ and morphisms the $\sim$classes of morphisms on pro-$\PP$ of $\PP$-expansions of objects of $\T$. Two objects are $X,Y\in\T$ are then {\em shape equivalent\/} if they have $\PP$-expansions $\XX$ and $\YY$ equivalent in the shape category.
\end{definition}

Note that any morphism in the shape category may be represented by a diagram
$$\begin{array}{ccc}
X&\buildrel \alpha\over\longrightarrow&\XX\\
&&\phantom{f}\big\downarrow f\\
Y&\buildrel \beta\over\longrightarrow&\YY
\end{array}$$
for some $\PP$-expansions $\alpha$ and $\beta$, and morphism $f$ in pro-$\PP$. Indeed any map $X\to Y$ in $\T$ gives rise to such a diagram, but the converse does not hold: a morphism $f\colon\XX\to\YY$ does not necessarily correspond to a map $X\to Y$ in $\T$.

We are particularly interested in {\em shape invariants}, invariants of objects in $\T$ which depend only on the shape equivalence class of the objects. The principal invariants we are concerned with here are \v{C}ech cohomology and (in the pointed version) the shape homotopy groups, defined respectively on an object $X\in\T$ with $\PP$-expansion $\XX$ by
$$\begin{array}{rl}
H^*(X)&=\displaystyle\lim_{\longrightarrow}\{H^*(X_0)\to\cdots\to H^*(X_{n-1})\to H^*(X_n)\to\cdots\}, \quad\mbox{and}\\
\pi^{\rm sh}_*(X,*)&=\displaystyle\lim_{\longleftarrow}\{\quad\cdots\to \pi_*(X_n,x_n)\to  \pi_*(X_{n-1},x_{n-1})\to\cdots \to \pi_*(X_0,x_0)\}\,.\end{array}$$
The usual property of \v{C}ech cohomology taking inverse limits of spaces to direct limits of cohomology groups means that the above coincides with the normal definition of \v{C}ech cohomology of a space $X\in\T$. A similar definition of $K$-theory for $X\in\T$ with $\PP$-expansion $\XX$ given by the direct limit $K^*(X)=\displaystyle\lim_{\lra}\{ K^*(X_n)\}$ can also be made, is a shape invariant and coincides with other appropriate forms of $K$-theory, for example that constructed from $C^*$-algebras.

An important class of objects for us are those spaces $X\in\T$ which are shape equivalent to objects in $\PP$. This is encapsulated in the following definition.

\begin{definition}\label{stable} A space $X\in\T$ or pointed space $(X,x)$ is
\emph{stable} if it is shape equivalent to a finite polyhedron.
\end{definition}

\begin{remark} A sufficient condition for the stability of a space $X$ is  that $X$ may be written (in $\T$) as an inverse limit
$$X=\lim_{\longleftarrow} \{\cdots\to X_n\buildrel f_n\over\longrightarrow X_{n-1}\to\cdots \to X_0\}$$
in which all the factor spaces $X_n$ are homotopy equivalent to finite polyhedra and all the bonding maps $f_n$ are homotopy equivalences.  In this case the homotopy and (co)homology groups associated to the $X_n$ `stabilise' and all the shape invariants of $X$ coincide with the corresponding invariants of each of the $X_n$ in this $\PP$-expansion.
\end{remark}

The final shape theoretic concept we will need will be that of {\em movability}. Borsuk~\cite{B} introduced the notion of movability for compact subspaces $X$ of the Hilbert cube $Q$ as in the following definition, but, for our work here, the properties discussed in the remaining results of this section form the more practical characterisation of this concept.

\begin{definition}
Say $X \subset Q$ is {\em movable\/} if for every neighbourhood $U$ of $X$ in $Q$ there is a neighbourhood $U_0\subset U$ of $X$ in $Q$ such that for every neighborhood $W \subset U$ of $X$ there is a homotopy $$H: U_0 \times I \to U$$ such that for all $x \in U_0,$ $H(x,0) =x$ and $H(x,1) \in W.$
\end{definition}

In other words, $U_0$ can be homotopically deformed  within $U$, i.e. ``moved," into a subspace of $W$. Every compact metric space can be embedded in $Q$ and the definition of movability can be reformulated to make no reference to an embedding of $Q.$ We refer the reader to~\cite{MS} for the proof of the equivalence of the various formulations of movability; see in particular ~\cite[Remark 2, p. 184]{MS}.

\smallskip
It is important to note the following relationship between stability and movability. See, for example, \cite{MS} for details.

\begin{theorem}\label{stabmove} A space is movable if it is stable, but the converse does not necessarily hold.
\end{theorem}

An informal, intuitive explanation for why stability as above implies movability is as follows. If a stable space $X$ is embedded in the Hilbert cube $Q$, then each projection $p_n\colon X \to X_n$ will extend (since each $X_n$ is an absolute neighborhood retract) to a neighborhood $\widetilde{p_n}\colon U_n \to X_n$ and one can choose these neighbourhoods to be decreasing to $X$, say $X=\cap U_n$ and $U_n \supset U_{n+1}.$ One can then ``move" a given $U_n$ into $U_{n+1}$ using the homotopy equivalence of the corresponding bonding map. In general, however, one can move $U_n$ into $U_{n+1}$ under weaker conditions.

\smallskip
However, we work primarily with a homological characterisation of movability, which will be more amenable than the definition above. First we recall

\begin{definition} The inverse sequence of groups and homomorphisms
$$\cdots\to A_2\buildrel a_2\over\lra A_1\buildrel a_1\over\lra A_0$$
satisfies the {\em Mittag-Leffler condition} {\em (ML)} if for each $n$ there is a number $N\geqslant n$ such that
$$\im(A_p\to A_n)=\im(A_q\to A_n)$$
for all $p,q>N$. Clearly if $p>q>n$, the image of $A_p$ in $A_n$ is contained in the image of $A_q$ in $A_n$; the system
is ML if the images of $A_p$ in $A_n$ are eventually constant for large values of $p$. The condition is obviously met in the case that
all the bonding homomorphisms $a_n$ are surjective.
\end{definition}

\begin{prop}\label{Move}\cite[Remark 3, p.$\,$184]{MS}
If the pointed space $(X,*)$ is movable and
$$\cdots\to(X_{n},x_{n}) \buildrel{f_n}\over\longrightarrow
(X_{n-1},x_{n-1})\to\cdots\to(X_0,x_0)$$ is \textit{any} inverse sequence of compact polyhedra with limit homeomorphic to $(X,*)$, then for each non-negative integer $k$ the resulting inverse sequence of homotopy
groups
$$\cdots\to\pi_k(X_{n},x_{n}) \buildrel{(f_n)_*}\over\longrightarrow
\pi_k(X_{n-1},x_{n-1})\to\cdots\to\pi_k(X_0,x_0)$$
satisfies the Mittag-Leffler condition.
\end{prop}


In the special case of one-dimensional continua, a converse also holds.

\begin{theorem}\label{1dMove}
If the pointed space $(X,*)$  is homeomorphic to the limit of the
inverse sequence of finite connected one-dimensional polyhedra
$\left((X_n,x_n);\:f_n\right)$, then $(X,*)$ is movable if and only
if
\[\{\pi_1(X_n,x_n); (f_n)_*\}
\]
satisfies the ML condition.
\end{theorem}

This follows from the stronger theorem~\cite[Theorem 4, II $\S$ 8.1,
p. 200]{MS}.

\begin{remark} We should point out that shape theoretic results stated
in terms of inverse sequences hold for \emph{any} shape expansion of
a space into an inverse system and include such expansions as the
\v{C}ech expansion~\cite[I,\S4.2]{MS} whose inverse limit is not
necessarily homeomorphic to the original space. Any representation
of a continuum as an inverse limit of finite CW-complexes does yield
a shape expansion~\cite[I,\S5.3]{MS} and since these are the
expansions readily available for tiling
spaces~\cite{AP,BD,BDHS,Ga,S}, we shall only state results in that
context.
\end{remark}

Movability and its characterisation in Theorem \ref{1dMove} are relevant to the understanding of the embeddings in surfaces in the light of the next result.

\medskip
\begin{theorem}~\cite[Theorem 7.2]{K}~\cite{M}\label{sufmov} If $X$ is a
subcontinuum of a closed surface and if $x$ is any point of $X$,
then $(X,x)$ is movable.
\end{theorem}

In fact, in his proof~\cite{K} Krasinkiewicz shows that any such
$(X,x)$ is shape equivalent to the wedge of finitely many circles or
the Hawaiian earring (with point given by the wedge point), but his
proof only treats the case that the ambient manifold is orientable.
This is a natural generalisation of the analogous result for
continua embedded in the plane obtained by Borsuk~\cite[VII,\S
7]{Btext}.


\subsection{Some homological algebra}

The identification of the Mittag-Leffler condition above being relevant to our discussion leads us to introduce some further homological algebra, culminating below in Theorem \ref{ShapeSurf}.

\begin{definition}
For an inverse sequence $\mathcal{A}$ of groups and homomorphisms

$$\cdots\to A_2\buildrel a_2\over\lra A_1\buildrel a_1\over\lra A_0$$

\noindent let the equivalence relation $\approx$ on $\prod_n A_n$ be given by
$(x_n) \approx (y_n)$ if and only if there is a $(g_n)\in \prod_n
A_n$ such that $(y_n)=(g_n \cdot x_n \cdot a_{n+1}(g_{n+1}^{-1}))$. Then
${\rm lim}\sp{1}\A$ is defined to be the pointed set of
$\approx$--classes with base point given by the class of the
identity element of $\prod_n A_n$.
\end{definition}

If $d: \prod_n A_n \rightarrow \prod_n A_n$ is given by
$d((x_n))=(x_n \cdot a_{n+1}(x_{n+1}^{-1}))$, then ${\rm
lim}\sp{1}\A$ is the trivial pointed set $\{*\}$ if and only if $d$ is onto. If
$d$ is a homomorphism (which is the case whenever each $A_n$ is abelian), ${\rm coker}\,d ={\rm lim}\sp{1}\A.$ In general, $\li\A$ is uncountable if it is not trivial.

\medskip

The following theorem follows  from general considerations, see,
e.g.,~\cite[Theorem 10, II $\S$ 6.2, p. 173]{MS}.

\begin{theorem}\label{MLlim} If the inverse sequence $ \A $
satisfies the ML condition, then $\li \A $ is trivial.
\end{theorem}

\medskip

In~\cite{G} Geoghegan shows the converse under a natural condition.
\medskip
\begin{theorem}\label{li1}~\cite{G} If each group $A_n$ in the inverse sequence $\A$ is countable and $\li \A$ is
trivial, then $\A$ satisfies the ML condition.
\end{theorem}

An advantage of  $\li \A$ over the ML condition is that it is more
amenable to calculation, as indicated by the following result that we shall use.

\medskip
\begin{lemma} ~\cite[Theorem 8, II $\S$ 6.2, p. 168]{MS} Given a short exact sequence of inverse systems of groups
$$1\to (A_n,a_n)\to (B_n,b_n)\to (C_n,c_n)\to 1,$$
that is, a commutative diagram

$$\begin{array}{rccccl}
&1&&1&&1
\\
&\big\downarrow&&\big\downarrow&&\big\downarrow
\\
\cdots\to&A_2&\buildrel a_2\over\lra&A_1&\buildrel a_1\over\lra&A_0
\\
&\big\downarrow&&\big\downarrow&&\big\downarrow
\\
\cdots\to&B_2&\buildrel b_2\over\lra&B_1&\buildrel b_1\over\lra&B_0
\\
&\big\downarrow&&\big\downarrow&&\big\downarrow
\\
\cdots\to&C_2&\buildrel c_2\over\lra&C_1&\buildrel c_1\over\lra&C_0
\\
&\big\downarrow&&\big\downarrow&&\big\downarrow
\\
&1&&1&&1
\end{array}$$
in which the columns are exact, there is an induced six term exact
sequence of pointed sets
\begin{equation}\label{6term}1\to\lim_{\longleftarrow}  A_n\to\lim_{\longleftarrow}  B_n\to\lim_{\longleftarrow}  C_n\to\lim\!\null^1 A_n\to\lim\!\null^1 B_n\to\lim\!\null^1 C_n\to1\,.\end{equation}

\end{lemma}
(An exact sequence of pointed sets satisfies the usual conditions
for an exact sequence of groups, where the kernel is understood to
be the pre-image of the base point of the pointed set.)

\medskip
Piecing together the above results we arrive at the following theorem.

\begin{theorem}\label{ShapeSurf} If $X$ is homeomorphic to the inverse limit of the
sequence of  finite polyhedra $\left((X_n,x_n);\:f_n \right)$ and if
$\li \left((\pi_1(X_n,x_n); (f_n)_* \right)$ is not trivial, then
$X$ cannot be embedded in a closed surface.\qed
\end{theorem}

\section{Tiling spaces and attractors}
\subsection{The space of an aperiodic tiling}\label{tilings}
For our purposes here, a {\em tiling\/} $P$ of $\R^d$ is a decomposition of $\R^d$ into a union of compact, polyhedral regions, each translationally congruent to one of a finite number of fixed {\em prototiles\/} and meeting only on their boundary, full face to full face. In the case $d=1$, a tiling is essentially equivalent to a bi-infinite word in a finite alphabet indexed by $\Z$: such a word determines the tiling combinatorially, and it is determined geometrically with  the additional information of the lengths of the individual prototiles and the relative position of $0$ in the tiling. The {\em topological\/} information we will associate to $P$, in particular the homeomorphism class of the {\em tiling space\/} $\Omega_P$ associated to $P$ (see definition \ref{tiling space} below) will depend only on such combinatorial information. See \cite{Sadbook} for a full discussion of the basics of tiling theory.

The tilings we have in mind will typically satisfy two further important properties. Here and elsewhere, let us  write $B_r(x)$ for the open ball in $\R^d$ of radius $r$ and centre $x$.

\begin{definition}\begin{enumerate}
\item A tiling $P$ of $\R^d$ is said to be {\em aperiodic\/} if it has no non-trivial translational symmetries, i.e., if $P=P+x$ for some $x\in\R^d$, then $x=0$.
\item We say $P$ is {\em repetitive\/} if, for every $r>0$, there is a number $R>0$ such that for every $x,y\in\R^d$, the patch $B_R(x)\cap P$ contains a translation of the patch $B_r(y)\cap P$.
\end{enumerate}\end{definition}

One of the aims of this paper is to identify, for each of our attractors $A\subset M$, a tiling $P$ whose associated tiling space $\Omega_P$ is homeomorphic to $A$. We now formally introduce this space, also known in the literature as the {\em continuous hull\/} of $P$. First however, we must describe the {\em local topology\/} on a set of tilings.

\begin{definition}
Suppose $\mathbb{W}$ is a set of tilings in $\R^d$. The  {\em local topology\/} on a $\mathbb{W}$ is given by the basis of open sets defined by all the {\em cylinder sets}. For $W\in\mathbb{W}$ and parameters $r, s>0$, define the cylinder set
$$U(W,r,s)=\left\{V\in\mathbb{W}\,|\, B_r(0)\cap  V=  B_r(0)\cap (W+x)\quad
\mbox{for some $x\in\R^d$ with $|x|<s$}\right\}\,.$$
That is, $U(W,r,s)$ consists of those tilings which agree with $W+x$ out to distance $r$ from the origin, for some translate $x$ of length less than $s$.
\end{definition}

This topology is metrisable, and the reader will find many sources (e.g.\ \cite{Sadbook}) which defines it directly  in terms of a specific metric $\partial$. Loosely speaking, the metric $\partial$ declares two tilings to be close if, after a small translation, they agree out to a large distance from the origin.

\begin{definition}\label{tiling space}
The  {\em tiling space\/} of $P$ is the space $\Omega=\Omega_P$ of all tilings $S$ of $\R^d$ all of whose local patches $B_r(x)\cap S$ are translation images of patches occurring in $P$, and  topologised with the local topology.
\end{definition}

Assuming $P$ is repetitive, $\Omega_P$ may also be defined as the completion of $P+\R^d$, the set of all translates of $P$, with respect to the metric $\partial$.

We shall meet  in Section \ref{OneD} the particular examples of tilings generated by substitutions. For this class of examples we shall give a further (equivalent) definition of the corresponding tiling space.

For a repetitive, aperiodic tiling $P$, the space $\Omega_P$ is compact, connected and locally has the structure of a Cantor set crossed with a $d$-dimensional disc; in fact it can be shown that, up to homeomorphism, $\Omega_P$ has the structure of a Cantor fibre bundle over a $d$-torus \cite{SW}.

A host of results \cite{AP, BD, BDHS, Ga, Kal, S}, {\em etc.}, variously identify a tiling space as an inverse limit of finite, path connected complexes. These results are applicable to tilings varying from the very general to specific classes, but one motivation for many of them has been to decompose the tiling spaces in such a way as to make computation of cohomology and $K$-theory accessible: if
$$\Omega_P=\lim_{\longleftarrow} \{\cdots\to X_n\to X_{n-1}\to\cdots\to X_1\}$$
then, for example, the \v{C}ech cohomology is computed as $H^*(\Omega_P)=\underrightarrow{\lim}\, H^*(X_n)$.


Although results like these show that the formalism of shape theory is very natural to apply to the subject of tiling spaces, it has not explicitly been done as far as we are aware. Nevertheless, it is interesting to note that several of the crucial steps in the papers such as \cite{BD, BDHS, Kal} which are particularly effective at computing cohomology, use essentially a shape equivalence: the machines developed compute the cohomologies $H^*(\Omega_P)$  by actually computing the cohomology of a space that is shape equivalent, but not homeomorphic, to $\Omega_P$.

We are now in a position to introduce our $L$-invariant mentioned in the introduction.

\begin{definition} Suppose $P$ is a tiling of $\R^d$, and $\Omega_P$ is its associated tiling space. Define $L(\Omega_P)$ to be $\lim^1\pi_1(X_n)$ for any $\PP$-expansion
$$\Omega_P=\lim_{\longleftarrow} \{\cdots\to X_n\to X_{n-1}\to\cdots\to X_1\}$$
with path connected, pointed complexes $X_n$.
\end{definition}

Following the discussion in the previous section, the invariant $L(-)$ takes values in the category of {\em pointed sets}. By construction, $L(-)$ is a shape invariant, and hence an invariant of $\Omega_P$ up to homeomorphism. It is in fact the first of a series of such invariants, and although we do not use them here, we record

\begin{definition} Suppose $P$ is a tiling of $\R^d$, and $\Omega_P$ is its associated tiling space. For $i\in\N$, define $L_i(\Omega_P)$ to be $\lim^1\pi_i(X_n)$ for any $\PP$-expansion
$$\Omega_P=\lim_{\longleftarrow} \{\cdots\to X_n\to X_{n-1}\to\cdots\to X_1\}$$
with path connected, pointed complexes $X_n$. Then $L_i(\Omega_P)=L(\Omega_P)$ for $i=1$, while for higher $i$ it will take values in abelian groups.
\end{definition}

\begin{remark}\label{homologylim1}
By the work of the previous section, the $L$-invariant for 1 dimensional tilings provides an obstruction to the tiling space being movable, and hence to it being realised as a subspace of a surface. In fact, homology or cohomology frequently suffice to determine that a space is not movable since, for $X$ a finite, path connected CW complex, and as $H_1(X)$ is the abelianisation of $\pi_1(X)$, if the inverse system
$$\cdots\pi_1(X_n)\to\pi_1(X_{n-1})\to\cdots\to \pi_1(X_0)$$
is  ML, then the system
$$\cdots H_1(X_n)\to H_1(X_{n-1})\to\cdots\to H_1(X_0)$$
is also ML. Thus the non-vanishing of $\lim^1 H_1(X_n)$  implies the non-vanishing of $\lim^1 \pi_1(X_n)$; similarly, divisibility in $\displaystyle{\lim_{\longrightarrow}} \,H^1(X_n)$ will also imply that the $L$-invariant is non-zero.

For example, consider the dyadic solenoid $\mathcal{S}$ given by the inverse limit of circles $X_n=S^1$ with bonding maps the doubling map. This space can be seen to be not movable from the fact that $\lim^1 H_1(X_n)$ does not vanish (it is a copy of the 2-adic integers mod $\Z$), or equivalently from the fact that $H^1(\mathcal{S})=\displaystyle{\lim_{\longrightarrow}} \,H^1(X_n)=\Z[\frac{1}{2}]$. However, we will see in Section~\ref{subtil} examples of tiling spaces for which the finer $L$ invariant and the associated homotopy groups are necessary to detect lack of movability.
\end{remark}

\subsection{Expanding attractors in codimension 1}\label{EAcodim1}

Recall that, given a diffeomorphism $h$ of a $C^k$-manifold $M$, $k\geqslant1$,  an
\emph{attractor} $A$ is a closed invariant set that admits a closed
neighborhood $N$ such that

 \begin{enumerate}
   \item $h(N) \subset \text{Interior}(N)$,
   \item $A$ consists of non-wandering points of $h$ and
   \item $A = \bigcap_{n\in \N}h^n(N).$
 \end{enumerate}

We will consider the case that $A$ is a continuum of codimension one in $M$ (i.e., it has topological dimension one less than that of $M$) and that  $A$ is an expanding attractor. Then each point $x \in A$ has a stable manifold
$W^s(x)=\{y\in M \,|\, \mathrm{dist}(h^n(x),h^n(y))\rightarrow 0
\:\text{ as } n \rightarrow \infty \}$ homeomorphic to $\R$ and an
unstable manifold $W^u(x)=\{y\in M \,|\,
\mathrm{dist}(h^n(x),h^n(y))\rightarrow 0 \:\text{as } n \rightarrow
-\infty \}$ homeomorphic to $\R^{d}$, both of which are
submanifolds of $M.$ For any given $x,y \in A$ we have $W^s(x) \cap
W^u(y)\subset A $ and at each point in this intersection the
corresponding tangent spaces of the stable and unstable manifolds
split the tangent space of $M$ into a direct sum. In the expanding
case under consideration, for each point $x \in A,$  $W^u(x) \subset
A$ while $W^s(x)$ intersects $A$ in a totally disconnected set.
Given points $x,y \in A$ and fixed orientations on $W^s(x)$ and $
W^u(y),$ if for each point $z\in W^s(x) \cap W^u(y)$ there is a
neighborhood $U$ of $z$ that can be oriented in a such a way that
its orientation coincides with the orientations induced by $W^s(x)$
and  $ W^u(y)$ at all points in $W^s(x) \cap W^u(y)\cap U, $ then
the attractor $A$ is said to be \emph{orientable}; otherwise, $A$ is
\emph{unorientable}. In~\cite{P1,P2} Plykin proved fundamental
theorems about the structure of such attractors that are essential
for our results and are summarized in ~\cite{P3}. Many of these
results relate to the structure of the restriction of the $h$ to
$W^s(A)=\cup_{x\in A}W^s(x)$, the \emph{basin of attraction} of $A$.

\smallskip
\begin{theorem}~\cite[2.2]{P3}\label{Por} If the continuum $A$ is an orientable codimension $1$
expanding attractor of the diffeomorphism $h$ of a $C^{k \geq
1}$-manifold $M$ of dimension $d+1\geq 3$, then $W^s(A)$ is
homeomorphic to a $(d+1)$--dimensional torus $\mathbb{T}^{d+1}$ with some
finite number $k$ points removed. Moreover, $W^s(A)$ can be
compactified by adding $k$ points to form a space
$\overline{W^s(A)}$ that is homeomorphic to $\mathbb{T}^{d+1}$ in such a
way that $h$ can be extended to a diffeomorphism
$\overline{h}:\overline{W^s(A)} \rightarrow \overline{W^s(A)}$ that
is topologically conjugate to a $DA$-diffeomorphism of
$\mathbb{T}^{d+1}$.
\end{theorem}

Recall that a $DA$-diffeomorphism of $\mathbb{T}^{d+1}$ is obtained by modifying an Anosov automorphism $\mathbb{T}^{d+1} \to \mathbb{T}^{d+1};$ that is,  an automorphism of $\mathbb{T}^{d+1} = \R^{d+1}/\Z^{d+1}$ that lifts to an automorphism of $\R^{d+1}$ represented by a matrix in $\mathrm{GL}(d+1,\Z)$ having no eigenvalues of modulus one. The modification takes the form of inserting a source along each of a finite number of periodic orbits of the automorphism. These  maps were first introduced by Smale~\cite[9.4(d)]{Sm} and are explained in detail in, for example, \cite[Chapt 8.8]{R}, \cite[Chapt 4.4, Ex. 5]{PdM}.

A classic example of a $DA$-diffeomorphism is derived from the automorphism $A$ of $\mathbb{T}^2$ represented by the matrix $\left( \begin{array}{rr}
1&1\\
1&0
\end{array}\right) $ modified at the fixed point $\mathbf{0}$ by changing the automorphism in a small disk $V$ containing $\mathbf{0}$ in its interior and leaving the automorphism unchanged outside $V.$ The derived diffeomorphism $h$ has a source at  $\mathbf{0}$ and is isotopic to the original automorphism $A,$ as can be seen by isotopically deforming the disk $V$ to a point.  The diffeomorphism $h$ can be made $C^{\infty}$ and has a one-dimensional attractor that is locally homeomorphic to the product of an interval and the Cantor set.

\medskip \medskip

Plykin also obtained a corresponding result for unorientable attractors.

\begin{theorem}~\cite[2.2]{P3}\label{Punor} If the continuum $A$ is an unorientable codimension $1$
expanding attractor of the diffeomorphism $h$ of a $C^{k \geq
1}$-manifold $M$ of dimension $d+1\geq 3$, then there is a manifold
$\widetilde{W^s(A)}$ and a commutative diagram

$$\begin{array}{rcl}\label{orcover}
\widetilde{W^s(A)}&\buildrel\tilde{h}\over\lra&\widetilde{W^s(A)}
\\
\big\downarrow \pi & & \big\downarrow \pi
\\
W^s(A)&\buildrel h\over\lra& W^s(A)
\end{array}$$
where $\tilde{h}$ is a diffeomorphism with an orientable expanding
attractor $\widetilde{A}=\pi^{-1}(A)$ and $\pi$ is a two--to--one
covering map.

\end{theorem}

\section{Attractors of dimension one}\label{OneD}
\subsection{The shape of a dimension 1, codimension 1 expanding attractor.}\label{shapeof1}

In this part we  prove the stability of 1 dimensional expanding attractors that embed in a surface, and in so doing prove Theorem \ref{thm1} in the case $d=1$.

Williams~\cite{W1,W2} showed that any one-dimensional expanding attractor is
homeomorphic to  the inverse limit space
$$A=\underleftarrow{\lim}\left(\bigvee^r S^1; s\right)\colon=
\underleftarrow{\lim}\left\{\cdots\to\bigvee^r S^1\buildrel s\over\longrightarrow
\bigvee^r S^1\buildrel s\over\longrightarrow
\bigvee^r S^1\to\cdots\to \bigvee^r S^1
\right\}$$
for
an expansion $s\colon (\bigvee^r S^1,p)\to (\bigvee^r S^1,p)$ on the one point
union of $r$ copies of the circle $S^1$ that fixes the wedge point $p.$
Notice that this is true independent of whether or not the attractor is
orientable.

Note that $\pi_1(\bigvee^r S^1,p)$ is the free group $F^r$ on $r$ letters, and, up to homotopy, the map $s\colon (\bigvee^r S^1,p)\to (\bigvee^r S^1,p)$ is determined by the endomorphism $s_*$ in $\pi_1(-)$, that is by the endomorphism $s_*\colon F^r\to F^r$.

\begin{lemma}\label{Shape} Suppose $A=\underleftarrow{\lim}(\bigvee^r S^1; s)$ for some map $s\colon (\bigvee^r S^1,p)\to (\bigvee^r S^1,p)$. Write $s_*$ for the corresponding
endomorphism of $F^r=\pi_1(\bigvee^r S^1,p)$ and $\mathcal{G}$ for the
resulting inverse sequence of groups. Then $A$ is stable if and only
if $\lim^1\,\mathcal{G}=1$.
\end{lemma}

\noindent{\bf Proof.} In general, any stable space is movable
(Theorem \ref{stabmove}, and see also~\cite[II $\S$ 8.1, p. 200]{MS}) and thus, by Proposition \ref{MLlim}, if $A$ is stable,  $\mathcal{G}$ is ML
and so $\lim^1\,\mathcal{G}=1$.

We prove the converse. Assume that $\lim^1 \mathcal{G}=1$. From
Theorem \ref{li1} we know that $\im s_*^n$ is eventually constant,
say $\im s_*^n=H\leqslant F^r$ for all $n\geqslant N$. Then $H$ is
necessarily a free group of some rank, $m$ say, where $m\leqslant
r$. We realise the inclusion $H\to F^r$ topologically as a map
$$j\colon (\bigvee^mS^1,q)\lra (\bigvee^rS^1,p)$$
(i.e., we take the $m$ generators of $H\leqslant F^r=\pi_1(\bigvee^rS^1)$, and represent them
as loops in $\bigvee^rS^1$; then $j_*$ in $\pi_1(-)$ realises the inclusion $H\to F^r$). Consider the diagram of spaces
\begin{equation}\begin{array}{ccc}
\bigvee^rS^1&\buildrel s\over\lra & \bigvee^rS^1\\
\big\uparrow j&&\big\uparrow j\\
\bigvee^mS^1 && \bigvee^mS^1\,.
\end{array}\end{equation}
As $\im s^n_*=\im s^{n+1}_*$ for all $n\geqslant N$, we can complete
the diagram with a map $\bigvee^mS^1 \buildrel w\over\lra
\bigvee^mS^1$ making the square commute up to homotopy and in particular inducing an isomorphism
$w_*\colon\pi_1(\bigvee^mS^1,q)\to\pi_1(\bigvee^mS^1,q)$. By the Whitehead
theorem, $w$ is a homotopy equivalence. We then have that $X$ is
shape equivalent to the inverse limit of spaces $\bigvee^mS^1$ and
bonding maps the homotopy equivalences $w$. Thus $X$ is shape
equivalent to $\bigvee^mS^1$.\qed

\begin{theorem}\label{1Dstable}
Any codimension one expanding attractor $A$ of a diffeomorphism of a surface is
stable.
\end{theorem}

\begin{proof}
By Williams' characterisation of one-dimensional attractors and the
above lemma, the stability of a one-dimensional expanding
attractor $A=\underleftarrow{\lim}(\bigvee^r S^1; s)$  is equivalent to the vanishing
of $\li$ for any associated inverse sequence of fundamental groups. However, this $\li$ must vanish by Theorem~\ref{ShapeSurf} since any subcontinuum of a surface is movable.
\end{proof}

\begin{remark}
Consider an attractor $A=\underleftarrow{\lim}(\bigvee^r S^1; s)$ with $r=1$. Due to the expansive nature of $s$, it will not induce an isomorphism on homology of $S^1$ and so the resulting attractor $A$ is not stable. Thus, any   expanding attractor of a diffeomorphism of a surface is shape equivalent to $\bigvee^rS^1$ for some $r>1$, which in turn is homotopy equivalent to a 2-torus $\TT^2$ with $r-1$ points removed. This proves Theorem \ref{thm1} for $d=1$. Note that this result does not require that the surface be orientable.
\end{remark}

\medskip

\subsection{Realising limit spaces as attractors.}\label{Realising}

We turn now to examine  conditions under which a space presented as a limit $\underleftarrow{\lim}(\bigvee^r S^1; s)$ can be realised as an expanding attractor for some diffeomorphism $h$ on a surface $M$. The question has two parts. From Theorem \ref{1Dstable}, a necessary condition is that $\underleftarrow{\lim}(\bigvee^r S^1; s)$ is stable, and we begin by considering conditions on the map $s$ that allow us to know when this is true, which, by Lemma \ref{Shape}, means conditions that tell us when the corresponding $\li (\im(s_*^n))$ vanishes. The second part, the construction of $M$ and $h$ when we know that $\underleftarrow{\lim}(\bigvee^r S^1; s)$ is stable, is addressed in part in the Remark \ref{BHremark} below, but in general this is a very difficult issue.

As the fundamental group of $\bigvee^r S^1$ is a free group $F^r$ on $r$ generators, we analyse the stability of $\underleftarrow{\lim}(\bigvee^r S^1; s)$ via the endomorphism $s_*\colon F^r\to F^r$. It is useful to consider also the abelianisation of this endomorphism, i.e., the corresponding endomorphism $s_*^{ab}$ and the commutative diagram
\begin{equation}\label{abel}\begin{array}{rcl}
F^r&\buildrel s_*\over\lra&F^r\\
\big\downarrow\pi&&\big\downarrow\pi\\
\Z^r&\buildrel s_*^{ab}\over\lra&\Z^r
\end{array}\end{equation}
where the vertical arrows $\pi$ are both abelianisation.

For convenience, given an inverse system of groups and endomorphisms
$$\cdots G\buildrel f\over\longrightarrow G\buildrel f\over\longrightarrow \cdots \longrightarrow G\buildrel f\over\longrightarrow G$$
we write for short $\li (f)^n$ for the corresponding $\li$ term.

We note the following simple but useful  condition, which is essentially a restatement of the observation in Remark \ref{homologylim1}.

\begin{lemma}\label{Hcond}
A necessary condition for $\li (s_*)^n$ vanishing is that $\li (s_*^{ab})^n$ vanishes. In particular, $\li (s_*)^n$ will not vanish unless $s_*^{ab}$ is projection onto a summand of $\Z^r$.\qed
\end{lemma}

\begin{remark}\label{simplify}
The image of $s_*$ is a free group on $t$ letters, where $t\leqslant r$. Without loss of generality, we shall assume that $\im s_*$ is of full rank, i.e., $t=r$, since if this is not so, then the rank of $\im s_*^n$ will eventually stabilise, say $\im s_*^n\cong F^k$ for large $n$, and instead of diagram \ref{abel} we can consider  the commutative diagram
\begin{equation}\label{abel}\begin{array}{rcl}
F^k&\buildrel s_*|_{{\rm Im} s_*^n}\over\lra&F^k\\
\big\downarrow\pi&&\big\downarrow\pi\\
\Z^k&\buildrel s|_*^{ab}\over\lra&\Z^k
\end{array}\end{equation}
where the rank of the top map, $s_*|_{{\rm Im} s_*^n}$ is of full rank (now $k$). As the towers
$$\cdots \to F^r\buildrel s_*\over\longrightarrow F^r\to\cdots\to F^r\qquad\mbox{and}\qquad
\cdots \to F^k\buildrel s_*|_{{\rm Im} s_*^n}\over\longrightarrow F^k\to\cdots\to F^k$$
are equivalent in the pro-category, the $\li$ term of one vanishes if and only if the $\li$ term of the other does.
\end{remark}

\begin{prop}\label{lim1cond}
Suppose $\im s_*$ is free of rank $r$.
\begin{enumerate}
\item Then $\li (s_*)^n$ vanishes if and only if $s_*$ is an isomorphism.
\item If $s_*^{ab}$ is not an isomorphism then $\li (s_*)^n$ does not vanish.
\end{enumerate}
\end{prop}

\begin{proof}
First note that $s_*$ is injective: as $\im s_*$ is free of rank $r$, we may regard $s$ as an epimorphism from $F^r$ onto a group isomorphic to $F^r$ (namely $\im s_*$). The Hopfian property of $F^r$ then tells us that $s$ is injective.

For (1), if $s_*$ is an isomorphism, then clearly the tower
\begin{equation}\label{FTower}\cdots \to F^r\buildrel s_*\over\longrightarrow F^r\to\cdots\to F^r
\end{equation}
is ML and $\li (s_*)^n=1$.

Conversely, if $s_*$ is not an isomorphism, then as it is injective, it must fail to be onto. Suppose $x\in F^r$ is not in the image of $s_*$. Then for each $n$, the element $s_*^{n-1}(x)$ in the image of $s_*^{n-1}$ is not in $\im s_*^n$, for if $s_*^n(y)=s_*^{n-1}(x)$ for some $y\in F^r$, by the injectivity of $s_*$, we have $s_*(y)=x$, contradicting the assumption on $x$. The sequence of sets $\{\im s_*^n\}$ is thus strictly decreasing with $n$ and tower (\ref{FTower}) is not ML. Hence $\li (s_*)^n\not=1$.

For (2), the case where $\im s_*^{ab}$ is of rank $r$ but $ s_*^{ab}$ is not an isomorphism is dealt with by Lemma \ref{Hcond}.

If $\im s_*^{ab}$ is of rank less than $r$, then by the commutativity of Diagram (\ref{abel}), the composite $\pi\circ s_*$ cannot be onto, and hence $s_*$ is not onto. The result now follows by the argument used in part (1).
\end{proof}

\medskip It is certainly not the case that an endomorphism $s_*\colon F^r\to F^r$  need be invertible for the corresponding inverse limit space to be stable, and the constructions of the Remark \ref{simplify} can be highly relevant.  The following example illustrates this point.

\begin{example}\label{silly3}
The endomorphism $s_*$ on $F^3$ with generators $a,b,c,$ given by
$$a \mapsto abc,\quad b\mapsto abc,\quad c\mapsto a$$
is not an isomorphism, but $\li (s_*)^n$ is trivial. This follows from the observation that the image of any power of $s_*$ is the free group $F^2$ generated by the two words $\alpha=a, \,\beta=abc$, and $s_*$ on $\im s_*$ acts as
$$\alpha\mapsto\beta \quad \beta\mapsto \beta\beta\alpha$$
which is invertible (as is its abelianisation). Thus the inverse system of groups is ML by part (1) of the Proposition.
\end{example}

The following example illustrates part (2) of the Proposition.

\begin{example}
Suppose $s_*$ is the endomorphism  on $F^2$ with generators $a,b$ given by
$$a \mapsto ababa,\quad b\mapsto baaab\,.$$
Then $s_*$ is of rank 2, but its abelianisation, given by the matrix
$\left(\begin{array}{cc}3&2\\3&2\end{array}\right)$,
 is of rank 1. It may also be readily checked that this $s_*$ is not invertible, and hence $\li (s_*)^n\not=1$.
\end{example}

\begin{remark}
Proposition \ref{lim1cond} reduces the question of the stability of $\underleftarrow{\lim}(\bigvee^r S^1; s)$ to questions about the ranks of $s_*$ and its abelianisation and a question about the invertibility of $s_*$; the latter being addressable by methods such as Stallings' folding technique. While in practice these criteria may or may not be easily addressed, our second major question, that of realising $\underleftarrow{\lim}(\bigvee^r S^1; s)$ as an attractor supposing we have established its stability, is a good deal harder.
\end{remark}

\begin{remark}\label{BHremark}
In general, the stability of a space of the form $\underleftarrow{\lim}(\bigvee^r S^1; s)$ alone is not sufficient to guarantee that it occurs as an attractor of a surface diffeomorphism. Given such a space which is stable, it remains to {\em geometrically realise\/} the map $s\colon \bigvee^r S^1\to \bigvee^r S^1$. Effectively, this means realising $\bigvee^r S^1$ as a subspace of a surface $M$, thickening it to a 2 dimensional neighbourhood $\bigvee^r S^1\subset N\subset M$ of the same homotopy type as $\bigvee^r S^1$ in such a way that $s_*\colon\pi_1(\bigvee^r S^1,p)\to\pi_1(\bigvee^r S^1)$ can be realised as the homomorphism $h'_*\colon\pi_1( N)\to\pi_1( N)$ of some (differentiable) embedding $h'\colon N\hookrightarrow N$ which also allows an extension to a diffeomorphism $h\colon M\to M$ of the whole surface. The space $\underleftarrow{\lim}(\bigvee^r S^1; s)$ is then homeomorphic to the attractor $\bigcap_{n\in\N} h^n(N)\subset M$. It is known that many autmorphisms are not geometrically realisable and in~\cite{Ge} it is even shown that in some sense most automorphisms $s_*$ of $F^r$ when $r>2$ are not geometrically realisable.

In~\cite[Theorem 4.1]{BH} Bestvina and Handel derive sufficient conditions in terms of a cyclic word  for an automorphism $\alpha$ of $F^r$ to be realisable in the above sense to a pseudo-Anosov automorphism of a surface with one boundary component. Given any such pseudo-Anosov automorphism, one can construct a derived from pseudo-Anosov automorphism of the associated closed surface (with no boundary) that has a one-dimensional attractor of the form $A=\underleftarrow{\lim}(\bigvee^r S^1; s)$ in a way that parallels the $DA$  automorphisms of the torus discussed in section \ref{EAcodim1}, where the mapping $s$ induces an automorphism of $\pi_1(\bigvee^r S^1, p)$ conjugate to $\alpha$.

The condition of~\cite[Theorem 4.1]{BH} does not apply to attractors in a closed surface with multiple components in its complement (which correspond to modifying the automorphism on more than one periodic orbit), and  finding general necessary and sufficient conditions seems quite difficult and will not be addressed here. The problem is made more complicated by the fact that the fundamental group itself does not uniquely determine surfaces with boundary, and some information about the boundary components must also be reflected in any sufficient conditions.
\end{remark}

\medskip
\subsection{One-dimensional orientable attractors and substitution tiling spaces}\label{subtil}

We turn to the issue of realising the {\em orientable\/} one dimensional attractors as tiling spaces, proving Theorem  \ref{thm2} for $d=1$. In contrast to the situation when $d>1$ that we will meet in the Section \ref{Higherd}, we can realise the one dimensional attractors as spaces of {\em primitive substitution\/} tilings, which we now introduce.

\begin{definition}
A one dimensional  \emph{substitution} is a function $\sigma$ from a finite alphabet $\A$ of at least two letters to the set $\A^*$ of non-empty, finite words composed of letters in $\A.$ Such a substitution is called \emph{primitive} if, given any pair $a,b$ of letters in $\A$, there is an $n$ such that the letter $b$ occurs in the word $\sigma^n(a).$
\end{definition}

Giving the set  $\A$ the discrete topology, the $\Z$--fold product $\A^{\Z}$ (with the product topology) is a Cantor set which supports the {\em shift homeomorphism\/} $S\colon \A^{\Z}\to \A^{\Z}$ that shifts the index of points in $\A^{\Z}$ by one: $S((x_i)) = (y_i),$ where $y_i=x_{i+1}.$

\begin{definition}
Given a substitution $\sigma$ on $\A,$ the \emph{substitution subshift $\Sigma$} associated to $\sigma$ is the subspace of all points $(x_i) \in \A^{\Z}$ satisfying the property that for all $i\in \Z$ and all $k\in \N,$ the word $x_ix_{i+1}\cdots x_{i+k}$ is a subword of $\sigma^n(a)$ for some $n \in \N$ and some $a \in \A.$ The \emph{substitution tiling space} $\Omega_{\sigma}$ is the suspension of the shift homeomorphism $S$ restricted to $\Sigma$, i.e., the space $\Sigma\times\R/\sim$ where $\sim$ denotes the equivalence relation which identifies $((x_i) , t)$ with $(S((x_i)),t+1)$ for all $(x_i)\in\Sigma$ and $t\in\R$.
\end{definition}

It may be shown that for a primitive substitution $\sigma$, the space $\Omega_\sigma$ coincides with the tiling space $\Omega$ of Section \ref{tilings} associated to any of the elements $(x_i)$ of $\Sigma$.

In ~\cite{BD1} Barge and Diamond show that any orientable one-dimensional expanding attractor is homeomorphic to either a solenoid or a substitution tiling space. We sketch a proof. Consider a one-dimensional attractor $A=\underleftarrow{\lim}(\bigvee^r S^1, s)$ as before satisfying the conditions of an elementary presentation in the sense of Williams \cite{W2}.  By the orientability of $A,$  we can cover $A$ by consistently oriented flow box neighborhoods. Choose such a covering.

Now choose a term $X_n=\bigvee^r S^1$ in the inverse sequence defining $A$ such that the pullbacks in $A$ under the projection $p_n:A \lra X_n$ of sufficiently small arcs in $X_n$ are each contained in a flow box neighborhood. This allows us to orient each circle in $X_n$ (and  in fact, in all $X_m$ for $m>n$) consistently with the orientation of $A.$  Now construct an alphabet $\A=\left\lbrace a_1,\dots,a_r \right\rbrace $ whose letters correspond to each oriented circle in $\bigvee^r S^1$ and  define a function  $\sigma$ from $\A$ into the set of non-empty finite words induced by $s:X_{n+1} \lra X_n$; each circle in $X_{n+1}$ is mapped to a finite, ordered sequence of circles in $X_n$. Moreover, $\sigma$ has as values non-empty words with only positive powers.

To elaborate, the map $s:X_{n+1} \lra X_n$ determines how the small neighborhoods given by the pullbacks of small neighborhoods determined by arcs in $X_{n+1}$ fit within the flow box neighborhoods determined by the pullbacks of arcs in $X_n,$ and the consistent orientation of $A$ then implies that $s$ must preserve the given orientation of the circles.

Then $\sigma$ is a substitution when $r>1$ and $A$ is a solenoid if $r=1.$ By Williams' construction  we may assume that  $s$ satisfies the flattening condition that some neighborhood of the wedge point $p$ is mapped by some power of $s$ to a set homeomorphic to an interval. This implies that some power of the substitution $\sigma$ is \emph{proper} in the terminology of~\cite{BD} and so  forces the border in the sense of \cite{AP}. By the machinery of \cite{AP},  $\Omega_{\sigma}$ is therefore homeomorphic to $A.$

\medskip\noindent{\bf Proof of Theorem \ref{thm2} for $d=1$.}
Suppose $A$ is an orientable codimension 1 attractor. By the argument of  Barge and Diamond sketched above, we can identify $A$ with a space $\Omega_\sigma$ which is either a tiling space or a solenoid. The latter we can rule out since by Theorem~\ref{ShapeSurf}, we know that $L(\Omega_\sigma)$ must vanish, which is not the case for a solenoid, as in the example of Remark \ref{homologylim1}.\hfill$\square$

\begin{example}\label{FibEx}
A simple but usefully explicit example is given by the $DA$-diffeomorphism of the torus mentioned in Section \ref{EAcodim1}, which is derived from the automorphism represented by the matrix $\left( \begin{array}{rr}
1&1\\
1&0
\end{array}\right) $ and  has an attractor that is homeomorphic to the tiling space of the Fibonacci substitution
$$a \mapsto ab,\quad b\mapsto a.$$
\end{example}

\medskip\subsection{Embedding one dimensional substitution tiling spaces in surfaces.}\label{sect44}
Given the realisation in the result above of each orientable one dimensional, codimension one attractor as the tiling space of a primitive substitution, we turn to the converse question of which aperiodic, primitive substitutions $\sigma$ have a tiling space $\Omega_{\sigma}$ that can occur as an expanding attractor of a surface diffeomorphism: how close is the correspondence between these two sets of objects?

Holton and Martensen show in~\cite{HM} that whenever $\Omega_{\sigma}$ can be embedded in a closed orientable surface, it can occur as an attractor of a surface diffeomorphism, so our question addresses also the apparently more general issue of when we can identify a one dimensional tiling space as a subspace of an orientable surface.

In \cite{HM}  a necessary condition for $\Omega_\sigma$ to be embedded in such a surface is given, the condition requiring that the {\em asymptotic composants\/} \cite{BD1} of $\Omega_{\sigma}$
must form $n$-cycles for an even integer $n$, and moreover that the sum of indices of the cycles in an essential embedding is equal to the Euler characteristic of the ambient surface.

\medskip
Theorem \ref{1Dstable} gives a rather different necessary condition on the realisation of a tiling space as an attractor of a surface diffeomorphism.

\begin{cor}
Given a non-periodic tiling of $\R$ with tiling space $\Omega$, a necessary condition for $\Omega$ to be realisable as an attractor of a surface diffeomorphism is that $L(\Omega)=1$.\qed
\end{cor}

In the case of a primitive substitution $\sigma$ on $\A=\left\lbrace a_1,\dots,a_r \right\rbrace$, we develop tools for identifying information about the set $L(\Omega_\sigma)$ from $\sigma$.

First, we recall that in \cite{AP} Anderson and Putnam construct (among other things) a model for $\Omega_\sigma$ in the case of a primitive substitution $\sigma$ on an alphabet $\A$ of $r$ letters which satisfies the property of {\em forcing the border}. In this model -- precisely that appealed to in the construction of Section \ref{subtil} -- the space $\Omega_\sigma$ is described as the inverse limit
$$\Omega_\sigma=\displaystyle\lim_{\longleftarrow}\left\{\cdots
\bigvee^rS^1\buildrel s\over\longrightarrow\bigvee^rS^1\to\cdots\to\bigvee^rS^1\right\}$$
for a self map $s$ which in $\pi_1(-)$ realises the substitution $\sigma$. We immediately have

\begin{prop}
In the case where $\sigma$ forces the border \cite{AP} we have $L(\Omega_\sigma)=\li(s_*)^n$.\qed
\end{prop}
The techniques such as those developed in the  Section \ref{Realising} give methods of deciding if this vanishes.

However, it may well be that  $\sigma$ does not force the border, and the situation is then more complex. A number of models for $\Omega_\sigma$ as an inverse limit of a single bonding map are available, for example the Anderson-Putnam complex of collared tiles \cite{AP}, or the Barge-Diamond model \cite{BD} consisting of a complex $B$ with self map $g$ and subcomplex $Y$ composed of the so-called gluing tiles. We consider this latter model. As before, $L(\Omega_\sigma)=\li (g_*)^n$, and in principle all the data needed to compute this set is contained in the substitution $\sigma$. The following, however, provides a convenient tool. Recall from \cite{BD} that collapsing the subspace $Y$ to a point yields the space  $\bigvee^rS^1$ of the previous construction, with the commutative diagram of self maps
$$\begin{array}{rcl}
B&\longrightarrow&\bigvee^rS^1\\
g\big\downarrow&&s\big\downarrow\phantom{s}\\
B&\longrightarrow&\bigvee^rS^1\,.\end{array}$$

\begin{prop}\label{noShape}
Suppose $Y$ is path connected. Then $L(\Omega_\sigma)\not=1$ if $\li(s_*)^n\not=1$.
\end{prop}

\smallskip
\noindent{\bf Proof.}
If $Y$ is path connected, then the quotient map $B\to \bigvee^rS^1$ induces a surjection $\pi_1(B)\to
\pi_1\left(\bigvee^rS^1\right)$; this follows by observing that we can take the individual $S^1$'s as generating loops of $\pi_1\left((\bigvee^rS^1\right)$, and, by choosing a suitable maximal tree in $Y$, each loop can be lifted. This will not be possible in general if $Y$ is not path connected.

Then we have an induced short exact sequence of groups and self maps
$$\begin{array}{ccccccccc}
1&\to &Q&\lra&\pi_1(B)&\lra&\pi_1\left((\bigvee^rS^1\right)&\to &1\\
&&
\phantom{g_*|_Q}\big\downarrow g_*|_Q&&\big\downarrow g_*&&\big\downarrow s_*&&\\
1&\to &Q&\lra&\pi_1(B)&\lra&\pi_1\left(\bigvee^rS^1\right)&\to &1\end{array}$$
yielding an exact sequence of inverse limits finishing
$$\cdots\lra\lim\!\null^1(\pi_1(B);g_*)\lra\lim\!\null^1\left(\pi_1\left(\bigvee^rS^1\right);s_*\right)\lra 1\,.$$
Hence if $\li(s_*)^n\not=1$ then $\li(g_*)^n\not=1$. \qed

\medskip
The following examples illustrate applications of this result. In each case it is straightforward to check that the complex $Y$ is path connected and then the non-vanishing of the respective  $\li(s_*)^n$ follows by Proposition \ref{lim1cond} since each substitution homomorphism $s_*\colon F^r\to F^r$ fails to be onto (in each case its abelianisation  is however an isomorphism).

\begin{example}
The substitution $\sigma,$ $$1\mapsto 1131,\;\;2\mapsto1231,\;\;3\mapsto 232$$ is mentioned in~\cite{HM} as
an example of a substitution whose tiling space $\Omega_{\sigma}$ cannot be embedded in an orientable surface and yet meets the condition of \cite{HM} on even cycles of asymptotic composants.  Direct computation shows that $\sigma$ has
a connected gluing cell subcomplex $Y$ and $\im s_*$ is of rank 3.
However, the Stallings' folding technique shows that the endomorphism $F^3\to F^3$ induced
by $\sigma$ is not onto. Thus by Proposition \ref{lim1cond} the
group $\li(s_*)^n\not=1$, and so the tiling space is not stable and hence cannot be embedded in a
surface.
\end{example}

\begin{example}
The substitution $$\phi^3\colon\quad a\mapsto abaab,\quad b\mapsto aba \quad \text{ (the cube of
the Fibonacci substitution)}$$ is invertible and its tiling space is
stable and occurs as an attractor for a $DA$--diffeomorphism of the
torus, as sketched in Example \ref{FibEx}. However, the substitution

$$ a\mapsto ababa \;\; b\mapsto baa$$
has the same abelianisation  and (as it has no $bb$) has cohomology
identical to that of $\phi^3$.  Yet, this substitution is
not invertible and its tiling space is not stable, as can be seen by
applying Proposition~\ref{noShape}. Hence, this tiling space cannot
be embedded in a surface.

In particular, we note that these two substitutions have identical cohomology (and hence $K$-theory), but are distinguished by the $L$-invariant.
\end{example}

\begin{example}
Revisiting Example \ref{silly3}, we note that  the substitution on $\A= \left\lbrace a,b,c \right\rbrace$ given by
$$a \mapsto abc,\:\:b\mapsto abc,\;\; c\mapsto a$$
is not invertible, but the $L$-invariant of the tiling space is trivial. In fact the pro-equivalence of corresponding inverse systems shows  its tiling space is homeomorphic to that generated by the invertible
substitution on $\mathcal{B}= \left\lbrace \alpha,\beta \right\rbrace$
$$\alpha \mapsto \beta,\;\;\beta\mapsto \beta\beta\alpha$$
whose tiling space is stable and occurs as an attractor for a
$DA$--diffeomorphism of the torus.
\end{example}

\section{Higher dimensional codimension one attractors}\label{Higherd}

In this section we examine the nature of continua $A$ that embed as
a codimension one expanding attractor of a diffeomorphism $h$ of a
closed manifold $M$ of dimension at least 3. We continue with the notation that the dimension of $M$ is $d+1$.

We prove Theorems \ref{thm1} and \ref{thm2} for $d>1$ in Corollaries \ref{prfthm1} and \ref{proj} below, utilising the work of Plykin detailed in Section \ref{EAcodim1} and the analysis of projection tiling spaces by Forrest, Hunton and Kellendonk in \cite{FHK1}. We conclude with some further results about the restrictions our results imply on the possible codimension one embeddings of tiling spaces.

\medskip
\subsection{The shape of codimension 1 attractors.} To prove
Theorem~\ref{Por}, Plykin constructs a pair of transverse foliations
of $\overline{W^s(A)}$, the compactification of the basin of
attraction $W^s(A).$ In one foliation $\mathcal{S}$ all leaves are
homeomorphic to $\R$ and include (sometimes as subsets) the stable
manifolds of the points of $A$. The other foliation $\mathcal{U}$ is
by leaves homeomorphic to $\R^d$ and among the leaves are the
unstable manifolds of points of $A$. The nature of closed manifolds
admitting such foliations leads to the conclusion that
$\overline{W^s(A)}$ is homeomorphic to $\mathbb{T}^{d+1}.$ In the
construction, the $k$ points $\{x_1,\dots,x_k\}$ that are added
become fixed points of a diffeomorphism  $\overline{h}$ of
$\overline{W^s(A)}$ which when restricted to $A$ coincides with $h$.
This diffeomorphism is then topologically conjugate to a
$DA$--diffeomorphism of $\mathbb{T}^{d+1};$ in particular, a
diffeomorphism of $\mathbb{T}^{d+1}$ that is derived from an expanding
automorphism $\alpha$ with $1$--dimensional stable and
$d$--dimensional unstable manifolds by introducing repelling
periodic points (sources) at a finite number of periodic orbits of
$\alpha.$ The points $\{x_1,\dots,x_k\}$ in the remainder of the
compactification $\overline{W^s(A)}$ correspond under the topological
conjugacy to the points at which a source has been added to
construct the $DA$--diffeomorphism.

\begin{theorem}\label{or} If the continuum $A$
occurs as a codimension one expanding attractor of a diffeomorphism
$h$ of a closed manifold $M$ of dimension $d+1\geqslant 2,$ then $A$ is
stable.
\end{theorem}
\begin{proof}
The result for $\dim M=2$ is covered by Theorem~\ref{1Dstable}. We first treat the case of
$\dim M \geqslant 3$ and $A$ is orientable when Theorem~\ref{Por} and the
related constructions apply. First we replace $\overline{h}$ by a
positive iterate $f$ that fixes all the points $\{x_1,\dots,x_k\}$
in the remainder of the compactification $\overline{W^s(A)}$ and
preserves the orientation at each fixed point in the remainder. We
first note that $\cap_{n \in \N}f^n(W^s(A))=A$ and that the pair of
transverse foliations $\mathcal{S}$ and $\mathcal{U}$ of
$\overline{W^s(A)}$ are constructed in such a way that they are
invariant under $\overline{h}$ and thus $f.$  At each point $x_i$
one can form a neighborhood $U_i$ admitting a homeomorphism $h_i$
onto $\R^d \times (-1,1),$ where

\begin{enumerate}
\item $U_i \cap A = \emptyset$
\item $h_i(x_i)=(\mathbf{0},0),$
  \item for each $x\in \R^d,$ the segment $x\times
  (-1,1)$ is contained in a leaf of $\mathcal{S}$ and
  \item for each $x \in (-1,1),$ the hyperplane
  $\R^d \times x$ is contained in a leaf of
  $\mathcal{U}.$

\end{enumerate}
Moreover, we choose our neighborhoods $U_i$ to be pairwise disjoint.
In fact, Plykin~\cite{P1} constructs such neighborhoods, but their
existence also follows from the usual construction
of foliation charts for the foliations $\mathcal{S}$ and $\mathcal{U}$ once these foliations are known to exist.

Now let $U = \cup_{i=1}^k \,U_i$. Then since $\cap_{n \in
\N}f^n(W^s(A))=A,$ we also have that $\cap_{n \in \N}f^n(N)=A,$
where $N=\overline{W^s(A)}-U.$ Let $V_n = f^n(U)$. For each $n\in N$, the set $V_n$ has as its connected
components the $k$ sets formed by the $f^n$ images of the sets $U_i.$ By construction, one can
isotopically deform each of the components of $V_n$ within $V_{n+1}$ to obtain the components of
$V_{n+1}$. To see this, first observe that $V_{n+1} \supset V_n$ since $f$ has a repelling fixed point at each
$x_i$. Next, observe that $f$ expands the central hyperplane in each component of $V_n$ given
by $f^n(h_i^{-1}(\R^d\times 0))$ while fixing $x_i,$ yielding the central hyperplane of the corresponding component of
$V_{n+1}.$  At the same time, $f$ maps the hyperplanes $f_n(h_i^{-1}(\R^d\times x))$
onto corresponding hyperplanes of $V_{n+1}.$ Thus, $V_n$ sits tamely within $V_{n+1}$ as the union of $k$ balls within $k$ larger
balls in a way that can be nicely parameterised. Thus, letting $W_n = \overline{W^s(A)}-V_n$ we see
that for each $n \in N$ the inclusion $W_{n+1} \hookrightarrow W_n$ is a homotopy equivalence. As the complement of $k$ open balls in $\mathbb{T}^{d+1},$  each $W_n$
is homotopy equivalent to a finite polyhedron. Thus, the
limit of the inverse sequence
$$\cdots \hookrightarrow W_2 \hookrightarrow W_1 \hookrightarrow W_0$$
is stable. But this inverse limit is homeomorphic to
$\cap_n\,W_n=A.$

\medskip

We next treat the case that $A$ is unorientable and make use of the
double covering $\pi :\widetilde{W^s(A)}\rightarrow W^s(A)$ as in
Theorem~\ref{orcover}.  First, as above we construct the
neighborhoods $\widetilde{V_n}$ and $\widetilde{W_n}$ for
$\tilde{h}.$ One of the important features of the covering $\pi$ is
that its extension to the compactification
$\overline{\widetilde{W^s(A)}}$ is conjugate to the identification
map of an involution $I$ of the torus $\mathbb{T}^{d+1}$ that has the
form of the composition of a translation and the map $x \mapsto -x
$. Moreover, the points $\{\tilde{x_1},\dots,\tilde{x_n}\}$ in the
remainder of the compactification $\overline{\widetilde{W^s(A)}}$
correspond to the fixed points of $I.$ We then replace the
neighborhoods $\widetilde{V_n}$ by saturated neighborhoods
$\widetilde{V'_n}$ of the same form satisfying the four above
conditions on $V_n$; that is, neighborhoods that are closed under
application of the involution $I,$ which is possible since $I$ fixes
the points corresponding to $\{\tilde{x_1},\dots,\tilde{x_n}\}.$ We
then form the neighborhoods
$\widetilde{W'_n}=\overline{\widetilde{W^s(A)}}-\widetilde{V'_n}$
which satisfy $\cap_n\, \widetilde{W'_n} = \tilde{A}$ as above.  As
$\pi \circ \tilde{h} = h \circ \pi,$  the sets
$W_n=\pi(\widetilde{W'_n})$ satisfy $\cap_n W_n = A.$ As the sets
$\widetilde{V'_n}$ were constructed to be saturated open sets, the
sets $V_n=\pi(\widetilde{V'_n})$ are open in $W^s(A)$ and the
isotopy deformations of $\widetilde{V'_n}$ into
$\widetilde{V'_{n+1}}$ can be constructed so that they are mapped by
$\pi$ to homotopy equivalences of $V_n$ and $V_{n+1}.$ Then just as
above we have that $A$ can be realised as an inverse limit of the
homotopy equivalences $W_{n+1} \hookrightarrow W_n.$ At each stage
$W_n$ will be a compact manifold with boundary and so still have the
homotopy type of a finite polyhedron~\cite{West}. Thus, the
conclusion that $A$ is shape equivalent to $W_n$ still holds and so
$A$ is stable.
\end{proof}

\medskip
An immediate corollary of this proof is the statement of Theorem \ref{thm1} for $d>2$:

\begin{cor} \label{prfthm1} If $A$ is a codimension 1 orientable attractor in $M$, then it is shape equivalent to a $(d+1)$-torus with some finite number of points removed, and, in the unorientable case, it has such a space as a double cover.
\end{cor}

\begin{remark} It should be noted that being stable is not a sufficient condition
for embedding as a codimension one attractor. In fact,
Kalugin~\cite{Kal} has shown that the Penrose tiling space is
stable (a fact also deducible from \cite{AP}), but by Corollary~\ref{codim=comdim} below it cannot be embedded as a codimension one attractor: in short, no tiling space resulting from a canonical projection scheme with internal space of dimension more than 1 has the right cohomology.
\end{remark}

\medskip
\subsection{Realising as projection tiling spaces.}
We turn now to the proof of Theorem \ref{thm2} for manifolds $M$ of dimension at least 3, that is, the realisation up to homeomorphism of an orientable attractor in $M$ as a tiling space.

By a {\em projection tiling\/} we shall mean a tiling constructed
from a  projection scheme, as described in \cite{FHK1,
FHK2, GHK1, GHK2}. In keeping with the notation of those papers, the
dimension of the euclidean space in which the tiling is constructed
(the {\em external space}) $E$ is $d$, and the {\em codimension} (the
dimension of the {\em internal space} $E^\perp$) is denoted $n$.  To avoid confusion with the codimension of the attractors considered, we shall call $n$ the {\em internal dimension\/} of the projection scheme. The data for  the projection scheme consists of the {\em total space\/} $\R^{n+d}=E\oplus E^\perp$ which contains an $n+d$ dimensional lattice
$\Gamma$ and an {\em acceptance domain\/} $K\subset E^\perp$, a non-empty compact set which is the closure of its interior. The projection scheme is called {\em canonical\/} if $K$ is the projection to the internal space of an $n+d$ dimensional cube, and otherwise we shall call it {\em generalised}. A related class of projection schemes are the so-called {\em almost canonical\/} ones, introduced in \cite{GHK2}, and to which our results below also apply.

Given such a projection scheme, a whole family of tilings may be constructed, all locally equivalent to each other; indeed in some sense the tiling space, again denoted $\Omega$, for any (and all) of them is as naturally defined in terms of the projection scheme data as by any individual tiling. We consider in particular those with internal dimension 1.

The case of a generalised projection tiling space for general $d$ but $n=1$ is studied in detail in ~\cite[Chapt. III]{FHK1}. There $K$ consists of a (countable) disjoint union of closed intervals, and the resulting tiling space is given by the $(d+1)$-torus $\TT^{d+1}=(E\oplus E^\perp)/\Gamma$ cut on a set of $E$-orbits generated by the image of the boundary points of $K$ in $\TT^{d+1}$. In particular, it is shown how using such a scheme one can construct any Denjoy--like example, obtaining a $\Z^d$ action on a Cantor set given by cutting open any countable number of orbits of a $\Z^d$ action by translations on the circle. These actions lead to generalised projection tiling spaces that include as special cases the orientable attractors described in Theorem~\ref{Por}. Hence, we have the following.

\begin{cor}\label{proj}
Any orientable codimension 1 attractor in a manifold of dimension at least 3 is homeomorphic to the tiling space of a generalised projection with internal dimension one.\qed
\end{cor}

This, together with the work of the previous section, proves Theorem~\ref{thm2} of the Introduction.

\medskip
We conclude with two further points about the realisability or unrealisability of attractors as tiling spaces, or tiling spaces as attractors.

\medskip First, the use of the generalised projection tilings of internal dimension 1 to realise all attractors in manifolds of dimension at least 3 might lead one to wonder if the same projection schemes could also realise all the 1 dimensional attractors in surfaces as well. This is not so.  As any such projection tiling can be embedded in a torus, this shows that an attractor $A$ that may occur only in surfaces of genus greater than one cannot be homeomorphic to a projection tiling space. See, for example, ~\cite{Rob} for a specific higher genus case worked out in detail. Thus we need a different supply of tilings, such as those considered in the previous section, than the projection tilings to account for all the surface cases.

\medskip Secondly, the shape equivalence of an attractor to a $(d+1)$-torus less $k$ points, or to a space of which that is a double cover, puts considerable constraints on realising a given tiling space as an attractor, especially as $d$ gets large. We illustrate this by considering the cohomology of such a space.

\begin{prop}\label{cohcalc} Suppose $A$ is a codimension 1 attractor of a manifold of dimension at least 3. Then if $A$ is orientable, its \v{C}ech cohomology is given by
$$H^p(A)=\left\{
\begin{array}{ll}
0&\mbox{ if }p>d,\\
\Z^{d+k}&\mbox{ if }p=d,\\
H^p(\TT^{d+1})=\Z^{{d+1}\choose p}&\mbox{ if }0\leqslant p<d\\
\end{array}\right.$$
for a finite positive integer $k$. In the unorientable case, the free part of $H^*(A)$ includes in these groups.
\end{prop}

\noindent{\bf Proof.} By Theorem \ref{or}, in the unorientable case we know that $A$ will have the same cohomology as $\TT^{d+1}\!\!-\{k\}$ for some finite positive $k$. The cohomology  follows from an elementary Mayer-Vietoris calculation of  $H^*(\TT^{d+1}\!\!-\{k\})$.

In the unorientable case, $A$ has a 2-fold cover
$(\TT^{d+1}\!\!-\{k\})\buildrel\pi\over\lra A$. The transfer map
$\tau^*$ of a covering \cite{bredon} in cohomology with
coefficients in a ring $R$ gives a diagram
$$H^p(A;R)\buildrel\pi^*\over\lra H^p(\TT^{d+1}\!\!-\{k\};R)\buildrel\tau^*\over\lra H^p(A;R)$$
whose composite is multiplication by 2. If we choose $R=\Q$ then
$\pi^*$ must be injective and the result follows.\qed

\medskip
So, we see for example that most canonical projection tilings (in fact all those of codimension 2 or more, such as the Penrose tiling) do not embed as codimension 1 attractors:

\begin{cor}\label{codim=comdim} Suppose the tiling space $\Omega$ of a dimension $d>1$, internal dimension $n$ canonical projection tiling has the shape of a finite polyhedron $P$, and suppose $P$ is homotopy equivalent either to a $(d+1)$-torus with $k$ points removed, $(\TT^{d+1}\!\!-\{k\})$, or else has  $(\TT^{d+1}\!\!-\{k\})$ as a 2-fold cover. Then $n=1$. In particular, no internal dimension $n>1$ canonical projection tiling space can appear as a codimension one attractor of a $C^1$ diffeomorphism of a closed manifold.
\end{cor}

To prove this we need:

\begin{lemma}\label{projH1}
For $\Omega$ the tiling space of a dimension $d$, internal dimension $n$
canonical projection tiling with finitely generated cohomology, the
group $H^1(\Omega)$ contains a free abelian subgroup of rank at
least $n+d$.
\end{lemma}

\medskip\noindent{\bf Sketch Proof.}
This follows from the work of  \cite{FHK1}, (see also \cite{GHK2} which covers the more general case of `almost canonical' projections), which we briefly review.

The work of \cite{FHK1} sets up a method to
compute the \v{C}ech cohomology $H^*(\Omega)$ via certain exact
sequences in group homology. In particular, there is a short exact
sequence of $\Gamma$ modules
$$0\lra C^n\lra C^{n-1}\lra C^{n-2}_0\lra 0$$
which induces a long exact sequence
\begin{equation}\label{gpcohLES}
\cdots\lra H_{s+1}(\Gamma;C^{n-1})\lra H_{s+1}(\Gamma;C^{n-2}_0)\lra
H_{s}(\Gamma;C^n)\lra \cdots\,.\end{equation} The groups and map
$H_{s+1}(\Gamma;C^{n-2}_0)\lra H_{s}(\Gamma;C^n)$ in this sequence
may be identified with the homomorphism $H^{d-s}(\TT^{n+d})\lra
H^{d-s}(\Omega)$ induced by the almost everywhere 1--1 surjection
$\Omega \to\TT^{n+d}$ to the $n+d$ torus.

The results of \cite{FHK1} (ch.\,IV, Thm.\,6.7,  ch.\,V, Thm.\,2.4)
tell us that if $H^*(\Omega)$ is to be finitely generated then $n$
must divide $d$. In this situation, it can be deduced that
$H_{s}(\Gamma;C^{n-1})$ is non-zero only for $s\leqslant{d\over
n}(n-1)$. The exact sequence \ref{gpcohLES} and these observations
taken together yield for cohomology in dimension 1 the sequence
$$0\lra H^1(\TT^{n+d})\lra H^1(\Omega)\lra\cdots$$
from which the lemma follows.\qed

\medskip\noindent{\bf Proof of Corollary \ref{codim=comdim}.}
By Proposition \ref{cohcalc} $H^1(A;\Q)$ is a vector space of dimension at most $d+1$. If $\Omega$ is the tiling space of a canonical projection scheme of internal dimension $n$, then $H^1(\Omega;\Q)=\Q^{n+d}$. Clearly $n=1$ is the only possibility for $\Omega$ being shape equivalent to $A$.\qed

\end{document}